\documentclass[a4paper, 12pt, titlepage, fleqn]{article}
\usepackage{times}
\usepackage{amsthm}
\usepackage{xcolor}
\usepackage{enumitem}
\newtheorem{theorem}{Theorem}
\newtheorem{corollary}{Corollary}
\newtheorem{lemma}{Lemma}
\newtheorem{proposition}{Proposition}

\theoremstyle{definition}

\newtheorem{remark}{Remark}
\newtheorem{example}{Example}
\usepackage{amssymb,amsmath}

\newcommand{\FF}{\mathbb{F}}

\newcommand{\Fq}{\mathbb{F}_q}
\newcommand{\Fqn}{\mathbb{F}_{q^n}}

\newcommand{\G}{\mathcal G}

\def\F{\mathbb{F}}

\def\Fq{{\mathbb{F}}_q}

\def\Aut{\mathrm{Aut}}

\def\dim{\mathrm{dim}}

\def\L{\mathcal{L}}

\def\Tr{\mathrm{Tr}}
\def\End{\mathrm{End}}
\def\tr{\mathrm{tr}}

\def\nul{\mathrm{null}}

\DeclareMathOperator{\modr}{~mod_r}

\DeclareMathOperator{\gcrc}{gcrc}
\DeclareMathOperator{\lclc}{lclc}

\newcommand{\npmatrix}[1]{\left( \begin{matrix} #1 \end{matrix} \right)}

\newcommand{\rank}{\mathrm{rank}}

\begin{document}
\baselineskip=16.3pt
\parskip=14pt
\begin{center}
\section*{A Characterization of the Number of Roots  of  Linearized and Projective Polynomials
in the Field of Coefficients}

{\large 
Gary McGuire\footnote{Research supported by Science
Foundation Ireland Grant 13/IA/1914.} and John Sheekey 
 \\
School of Mathematics and Statistics\\
University College Dublin\\
Ireland}
\end{center}

\subsection*{Abstract}

A fundamental problem in the theory of linearized and projective polynomials over finite fields
is to characterize the number of roots in the coefficient field directly from
the coefficients. We prove results of this type, of a recursive nature.
These results follow from our main theorem which 
characterizes the number of roots using the rank of a matrix that is
smaller than the Dickson matrix.



\section{Introduction}

Linearized polynomials with coefficients in a finite field $\Fqn$ arise in many different
problems and contexts. 
A fundamental problem in the theory of linearized polynomials is
to characterize precisely the number of roots in $\Fqn$
in terms of the coefficients of the given polynomial.

More precisely, let 
$L(x)=a_0x+a_1x^q+a_2 x^{q^2}+\cdots+a_d x^{q^d}$
be a $q$-linearized polynomial with coefficients in $\Fqn$.
The roots of $L(x)$ that lie in the field $\Fqn$ form an $\Fq$-vector space,
which can have dimension anywhere between 0 and $d$.
The fundamental problem of linearized polynomials is to 
somehow determine
the dimension of this $\Fq$-vector space directly from the coefficients 
$a_0, a_1, \ldots ,a_d$. 
The ultimate goal is to find necessary and sufficient conditions on the $a_i$ that tell when
the dimension is equal to $k$, for each $k$ between 0 and $d$.
In this paper we present results of this type, and our approach
will generate further such results.
Here is one example of our results (see Theorem \ref{degree2alln}
for the full statement which covers all possibilities).

\begin{theorem}
Let $L(x) = ax+bx^q+cx^{q^2}\in \Fqn[x]$. Then 
$L(x)$ has $q^2$ roots in $\Fqn$ if and only if $G_{n-1}=0$ and $N(b/c)G_n = 1$,
where the quantities $G_n$ and $G_{n-1}$ are calculated directly from the coefficients
$a,b,c,$ via a recursion, and
$N$ is the norm from $\mathbb{F}_{q^n}$ to   $\mathbb{F}_{q}$.
\end{theorem}

For example, when $n=5$ this theorem becomes the following.

\begin{theorem}
Let $L(x) = ax+bx^q+cx^{q^2}\in \mathbb{F}_{q^5} [x]$. 
Let $u = \frac{a^q c}{b^{q+1}}$.
Then 
$L(x)$ has $q^2$ roots in $\mathbb{F}_{q^5}$ if and only if $b\ne 0$,
$(1-u)^{q^2+1}-u^q=0$, and $N(1-u)= N(c/b)$,
where $N$ is the norm from $\mathbb{F}_{q^5}$ to   $\mathbb{F}_{q}$.
\end{theorem}

We have several results of this flavour, which are presented in 
Sections \ref{sectd2} and \ref{sectd3}.
Our proofs use a relationship between the number of roots (or the nullity) of $L$ 
in $\Fqn$ and the rank of
a matrix we denote $A_L$, which can be calculated explicitly from the coefficients of $L$.
This relationship is established in Theorem \ref{thm:nullAL} in Section \ref{sectmain}
and is the main result of the paper.
It is known that the nullity of a linearized polynomial is equal to the nullity of a Dickson matrix
(see \cite{LN}).
However, the Dickson matrix can be difficult to calculate and work with.
In this paper we show that 
the nullity of a linearized polynomial is equal to the nullity of 
$A_L-I$, a matrix which is different to and smaller than the Dickson matrix.

An interesting feature of our results is that they are recursive in $n$, the degree of the extension
of $\Fq$.  This means that all the expressions in the coefficients for $\Fqn$ that occur
in our investigations 
(for example, the entries of $A_L$) will also
be used for $\F_{q^{n+1}}$ and higher degree extensions.

Alongside linearized polynomials we also consider projective polynomials,
which are polynomials of the form 
\[
\sum_{i=0}^d a_i y^{\frac{q^i-1}{q-1}}, \quad a_i\in \Fqn.
\] 
We will prove similar results for projective polynomials, giving necessary and
sufficient conditions on the coefficients for each possibility for
the number of roots in $\Fqn$.
The number of roots of a projective polynomial is related to the eigenvectors of 
our matrix $A_L$, see Section \ref{sectmain}.

A polynomial in $\Fqn [x]$ is called a \emph{permutation polynomial} if it induces
a bijective function on $\Fqn$.
A linearized polynomial is a permutation polynomial if and only if its only root in $\Fqn$ is 0.
Therefore, giving if and only if conditions on the coefficients of a
 linearized  polynomial to be a permutation polynomial can be seen as a special case of our results,
namely, the case when the dimension of the vector space of roots in $\Fqn$ is 0.
For a detailed statement about $ax+bx^q+cx^{q^2}$ see Theorem \ref{degree2alln}
and for $ax+bx^q+cx^{q^2}+dx^{q^3}$ 
see Theorem \ref{d3split3}.

\bigskip

Listing all previous work on this topic is impossible in a few paragraphs.
We give a few references here, and others throughout the paper.
Abhyankar \cite{Abhyankar} studied projective polynomials for a different reason
(realizing Galois groups).
Bluher \cite{Bluher} studied projective polynomials $x^{q+1}+ax+b$ and showed that 
the number of roots in the ground field is highly restricted.
Helleseth and Kholosha \cite{HellKhol} studied certain projective and linearized polynomials
in even characteristic, and gave a criteria for the number of roots in 
terms of the coefficients. Our results generalize and extend theirs
to all characteristics, and to higher degree.
The number of roots of a projective polynomial was related to the eigenvectors of a certain matrix by von zur Gathen-Giesbrecht-Ziegler \cite{VoGiZi}, however their paper
was not constructive, and does not give criteria based on the coefficients.

The paper is laid out as follows. In Section 2 we give preliminaries
and background for the paper.
Then in Section \ref{sectmain} we present our main results
relating the matrix $A_L$ to the number of roots.
Section \ref{sectd2} applies the main results to the case
of degree $q^2$, and Section \ref{sectd3} does the same for degree $q^3$.
Finally, in Section \ref{sectapp}
we outline some possibilities for future work.

\section{Preliminaries}

Throughout this paper,  $\Fq$ is a finite field with $q$ elements, 
$\Fqn$ is an extension field of degree $n$, 
and $\sigma$ is an automorphism of $\Fqn$ with fixed field $\Fq$; in other words 
$\sigma$ is a generator of the Galois group $\mathrm{Gal}(\Fqn:\Fq)$. 
We write $x^\sigma$ for the image of $x$ under $\sigma$.
Note that $x^\sigma = x^{q^s}$ for some $s\in \{1,\ldots,n\}$ with $\gcd(n,s)=1$.
We use $N$ to denote the norm function from $\Fqn \longrightarrow \Fq$, i.e.,
$N(x)=x^{1+\sigma + \sigma^2 + \cdots + \sigma^{n-1}}$ for $x\in \Fqn$.

\subsection{Linearized Polynomials}

A {\it $\sigma$-linearized polynomial} is a polynomial of the form 
\begin{equation}\label{lin1}
L=a_0x+a_1x^\sigma+a_2x^{\sigma^2}+\cdots+a_d x^{\sigma^d} \in \Fqn[x].
\end{equation}
If $d$ is the largest integer such that $a_d\ne 0$, we call $d$ the {\it $\sigma$-degree} of $L$, and denote it $\deg_\sigma(L)=d$. The $\sigma$-linearized polynomials are precisely those which define an $\Fq$-linear map from $\Fqn$ to itself.
As an $\Fq$-linear map, a $\sigma$-linearized polynomial has a rank and a nullity, 
and as usual $\rank(L) + \nul (L) = n$. Throughout, when referring to the {\it roots} of $L\in \L$, we will mean only elements of $\Fqn$; in other words, the kernel of $L\in \L$ as an $\Fq$-linear map on $\Fqn$. When referring to the nullity of $L\in \L$ we will always mean the nullity of the element of $\End_{\Fq}(\Fqn)$ defined by $L$.

Throughout the paper we will assume that $a_0\not= 0$. Although this is not usually necessary, some of the arguments 
have to be done differently when $a_0=0$, so for ease of presentation we will make this assumption.
When $a_0=0$,  one can usually apply a power of $\sigma$ 
to $L$ to obtain a polynomial that has $a_0\not= 0$.

Let $\L$ denote the set of all $\sigma$-linearized polynomials, and $\L_k$ the set of all $\sigma$-linearized polynomials of $\sigma$-degree at most $k-1$. The set of $\sigma$-linearized polynomials with addition and composition form a ring, which we denote by $(\L,\circ,+)$. If $M=b_0x+b_1x^\sigma+\cdots+b_e x^{\sigma^e}$, then
\[
L\circ M = \sum_{k=0}^{d+e} \left(\sum_i a_i b_{k-i}^{\sigma^i}\right) x^{\sigma^k}.
\]
If $L = L_1\circ L_2$ for $L_1,L_2\in \L$, we say that $L_2$ is a {\it right component} of $L$, and $L_1$ is a {\it left component} of $L$. We say that $L$ is a {\it left composition} of $L_2$ and a {\it right composition} of $L_1$.

We define the {\it greatest common right component} and {\it least common left composition} of $L,M\in \L$ in the natural way, and denote them by $\gcrc(L,M)$ and $\lclc(L,M)$ respectively. This terminology follows \cite{VoGiZi}.

\begin{remark}
The ring $(\L,\circ,+)$ is isomorphic to the {\it skew-polynomial ring} $\Fqn[t;\sigma]$; the non-commutative polynomial ring defined by $t\alpha = \alpha^\sigma t$ for all $\alpha\in \Fqn$. Much of the following can also naturally be expressed in terms of skew-polynomials. Though we will not use this in this paper, we will make use of some basic facts from the theory of skew-polynomials.
\end{remark}

\begin{remark}
In some literature, the symbol $\otimes$ is used instead of $\circ$, and the term {\it symbolic divisor} is used instead of component.
\end{remark}


The following are useful well-known facts. We refer the reader to 
\cite{Giesbrecht}, \cite{Jacobson}, and \cite{WuLiu}. Most are generalisations of results originally proved by Ore \cite{Ore} for the case $q$ prime.

\begin{proposition}
\label{prop:known}
Let $(\L,\circ,+)$ be the $\sigma$-linearized polynomial ring over $\Fqn$ for some $\sigma\in \Aut(\Fqn/\Fq)$.
\begin{enumerate}[label=(\roman*)]
\item
The multiplicative identity of $\L$ is $x$, and the centre of $\L$ is generated (as an $\Fq$-algebra) by 
$x$ and $x^{\sigma^n}$.
\item 
$(\L,\circ,+)$ is a right-Euclidean domain with respect to $\sigma$-degree.
\item
For any $L,M\in \L$, it holds that  $\deg_\sigma(L\circ M) = \deg_\sigma(L)+\deg_\sigma(M)$.
\item
(\cite{Jacobson} Propositions 1.2.11, 1.3.1) For any $A,B\in \L$, it holds that 
\[
\deg_\sigma(L)+\deg_\sigma(M)=\deg_\sigma(\gcrc(L,M))+\deg_\sigma(\lclc(L,M)).
\]
\item
The element $x^{\sigma^n}-x$ generates a maximal two-sided ideal in $\L$, and
\[
\frac{\L}{(x^{\sigma^n}-x)} \simeq M_{n\times n}(\Fq).
\]
\item
The {\it nullity} of $L\in \L_n$ is equal to $\deg_\sigma(\gcrc(L,x^{\sigma^n}-x))$.
\item
The nullity of $L\in \L$ is equal to the nullity of its associated {\it Dickson matrix}
\[
\mathcal{D}_L := \npmatrix{a_0&a_1&\cdots&a_{n-1}\\a_{n-1}^{\sigma}&a_0^{\sigma}&\cdots&a_{n-2}^{\sigma}\\\vdots&\vdots&\ddots&\vdots\\a_1^{\sigma^{n-1}}&a_2^{\sigma^{n-1}}&\cdots&a_0^{\sigma^{n-1}}}.
\]
\end{enumerate}
\end{proposition}

By $M \modr L$ we mean the unique $R\in \L$ of $\sigma$-degree less than $\deg_\sigma(L)$ such that there exists $Q\in \L$ satisfying 
\[
M = Q\circ L+R.
\]

\begin{remark}\label{gcdrem}
If $x^\sigma = x^q$, then it is known that $M \modr L$ is equal to the usual polynomial remainder $M\mod L$, although the quotients $Q,Q'$ are not equal; indeed if $M = Q'L+R$, $Q'$ is not in $\L$. Furthermore, the greatest common right component of $L,M\in \L$ coincides with the greatest common divisor (in the usual sense) of $L$ and $M$. 

Suppose now $x^\sigma=x^{q^s}$. A $\sigma$-linearized polynomial with coefficients in $\Fqn$ can also be viewed as a $\sigma$-linearized polynomial with coefficients in $\FF_{q^{ns}}$. Suppose $L,M$ are two $\sigma$-linearized polynomials with coefficients in $\Fqn$. 

As both the usual Euclidean algorithm with multiplication and the right-Euclidean algorithm with composition do not require any calculation in an extension field, the $\gcd$ (resp. $\gcrc$) of $L$ and $M$ over $\Fqn$ coincides with the $\gcd$ (resp. $\gcrc$) of $L$ and $M$ over $\FF_{q^{ns}}$. These coincide with each other by the previous statement, and so for any $\sigma$-linearized polynomials, their $\gcd$ and $\gcrc$ over $\Fqn$ coincide.
\end{remark}

\begin{example}
Let $n=5$, and let $x^\sigma=x^q$, $x^\tau=x^{q^2}$, and let $\L,\L'$ be the rings of $\sigma$- and $\tau$-linearized polynomials over $\FF_{q^5}$ respectively. Then the polynomial $L= x^{q^2}-x$ has $\sigma$-degree two, and $\tau$-degree one. Clearly $L$ has nullity one, as there are $q$ roots of $L$ in $\FF_{q^5}$. 

Then $\gcrc(x^{q^5}-x,x^{q^2}-x)= \gcd(x^{q^5}-x,x^{q^2}-x)=x^q-x$, which has $\sigma$-degree one. Now $x^{\tau^n}-x=x^{q^{10}}-x$, and $\gcrc(x^{q^{10}}-x,x^{q^2}-x)= \gcd(x^{q^{10}}-x,x^{q^2}-x)=x^{q^2}-x$, which has $\tau$-degree one also.
This illustrates Proposition \ref{prop:known} part (vi).
\end{example}

\subsection{Projective Polynomials}

Given $L\in \L$ of the form \eqref{lin1}, we define its associated {\it projective polynomial}
to be 
\[
P_L(y) = a_0+a_1y +a_2y^{[2]}+\cdots+a_dy^{[d]},
\]
 where
\[
y^{[i]} := y^{\frac{\sigma^i-1}{\sigma-1}}.
\] 
Note that $L(x) = xP_L(x^{\sigma-1})$, and so we say that $P_L$ is {\it associated} to $L$. We also say that $P_L$ has $\sigma$-degree $d$.

Projective polynomials were introduced by Abhyankar \cite{Abhyankar}, where the coefficient field was the transcendental extension $\Fq(t)$ of $\Fq$. The original motivation was to find polynomials with a given Galois group. However they have also been studied over $\Fqn$, for example by Bluher \cite{Bluher}, and in \cite{VoGiZi}, due to their interesting possible number of roots, and their connection to calculating composition collisions. Projective polynomials with maximum number of roots have been used in \cite{crypto} for attacking the 
discrete logarithm problem in cryptography; see Section \ref{sectapp} for further details.

The following is the most general known result for the number of roots of a projective polynomial.

\begin{theorem}[\cite{VoGiZi}]
The number of roots of a projective polynomial of $\sigma$-degree $d$ over $\Fqn$ is of the form
\[
\sum_i \frac{q^{n_i}-1}{q-1}
\] 
for some nonnegative integers $n_i$ with $\sum_i n_i = d$.
\end{theorem}

However this result does not give a way of determining the number of roots of a particular projective polynomial from its coefficients. In \cite{HellKhol}, Helleseth-Kholosha gave criteria for the case $d=2$, $q$ even. In this paper we will extend this to $q$ odd and general $d$, and we give a way
of determining the number of roots from the coefficients.

\subsection{Companion Matrix}\label{comp}

Let $L=a_0x+a_1x^\sigma+\cdots+a_d x^{\sigma^d}$ 
be a $\sigma$-linearized polynomial in $\Fqn[x]$ with $\sigma$-degree equal to $d$.
Define the {\it companion matrix} of $L$ (and of $P_L$) as the $d\times d$ matrix
\[
C_L = \npmatrix{0&0&\cdots&0&-a_0/a_d\\1&0&\cdots&0&-a_1/a_d\\\vdots&\vdots&\ddots&\vdots&\vdots\\0&0&\cdots&1&-a_{d-1}/a_d}.
\]
We define further the matrix
\begin{equation}
\label{eqn:AL}
A_L = C_L C_L^{\sigma}\cdots C_L^{\sigma^{n-1}},
\end{equation}
and we denote the characteristic polynomial of $A_L$ by $\chi_L$. 
In this paper we consider $A_L$  for the following reason.
\begin{proposition}
\label{prop:AL}
With the above notation, the map
\begin{align*}
\psi : & \ \L_d \longrightarrow \L_d \\
 & M  \mapsto  x^\sigma \circ M \modr L
\end{align*}
defines an $\Fqn$-semilinear map on $\L_d$ with associated matrix $C_L$ and automorphism $\sigma$.

The map
\begin{align*}
\phi : & \ \L_d \longrightarrow \L_d \\
 & M  \mapsto  x^{\sigma^n} \circ M \modr L
\end{align*}
defines an $\Fqn$-linear map on $\L_d$ with associated matrix $A_L$.
\end{proposition}

\begin{proof} 
Each map is well-defined, as $\L$ is a right-Euclidean domain. Furthermore both are additive, as $\L$ is a ring. 
Thus, to show semilinearity/linearity 
it suffices to show that $\psi(aM) = a^\sigma \psi(M)$ and $\phi(aM) = a\phi(M)$ for all $M\in \L$, $a\in \Fqn$. This follows from the fact that $x^\sigma \circ ax = a^\sigma x^\sigma$ for all $a\in \Fqn$, and so $x^{\sigma}\circ (aM) = a^\sigma x^{\sigma}\circ M$. Thus $x^{\sigma}\circ (aM)\modr L = a^\sigma (x^{\sigma}\circ M\modr L)$, as claimed. 

We choose the canonical basis $\{x,x^\sigma,\ldots,x^{\sigma^{d-1}}\}$ for $\L_d$, and represent $M\in \L_d$ as a column vector consisting of its coefficients. 
Then $x^\sigma\circ x^{\sigma^i} = x^{\sigma^{i+1}}$ for all $i$. 
Thus $\psi$ maps the $i$-th basis vector to the $(i+1)$-st basis vector for $i<d-1$. Furthermore 
\begin{align*}
x^\sigma \circ x^{\sigma^{d-1}} &= x^{\sigma^d}\\
& \equiv \frac{1}{a_d}(a_d x^{\sigma^d}-L) \modr L\\
&= -\frac{1}{a_d}\left(a_0x+a_1x^{\sigma}+\cdots+a_{d-1}x^{\sigma^{d-1}}\right),
\end{align*}
showing that $\psi$ has matrix $C_L$ with respect to the canonical basis.

Finally, as $\phi = \psi^n$, it follows that $\phi$ is a linear transformation of $\L_d$, with associated matrix $A_L$ with respect to the canonical basis of $\L_d$.
\end{proof}

\begin{remark}
In \cite{LeBorgne} the characteristic polynomial of the matrix $A_L$ is referred to as the {\it semi-characteristic polynomial} of the semilinear transformation $\psi$.
In \cite{S2} a similar matrix was used to construct new semifields and MRD codes.
\end{remark}

\begin{remark}\label{detnorm}
Note that $\det(A_L) = N(\det(C_L)) =  N((-1)^d a_0/a_d)$.
In particular, this implies that $\det(A_L)$ lies in $\Fq$.
In fact, all the coefficients of $\chi_L$ (the characteristic polynomial of $A_L$) are in $\Fq$,
as we now prove.
\end{remark}

\begin{theorem}\label{lem:charq}
The coefficients of the characteristic polynomial of $A_L$ all lie in $\Fq$.
\end{theorem}

\begin{proof}
The coefficients of the characteristic polynomial of $A_L$ are 
the elementary symmetric polynomials in the eigenvalues of $A_L$. Since the trace of $A_L^i$ is equal to the sum of the $i$-th powers of the eigenvalues of $A_L$, by Newton's identities it is sufficient to show that the trace of $A_L^i$ lies in $\Fq$ for each $i\geq 0$.

We have that $\Tr(A_L^i) = \Tr((C_LC_L^\sigma\cdots C_L^{\sigma^{n-1}})^i)$. 
By the well-known commutativity property of the trace function and by applying the automorphism $\sigma$
it is clear that $\Tr(A_L^i)^\sigma = \Tr((A_L^i)^\sigma) = \Tr(A_L^i)$, as claimed.
\end{proof}

\begin{remark}
We note that the matrix $A_L^i$ would be the matrix to use in place of $A_L$
if we considered $L\in \L$ on $\mathbb{F}_{q^{in}}$.
This is because 
$C_L$ has entries in $\Fqn$ and so
$$C_LC_L^\sigma\cdots  C_L^{\sigma^{in-1}}=
(C_LC_L^\sigma\cdots C_L^{\sigma^{n-1}})^i.$$
\end{remark}

\subsection{A useful homomorphism}

Given $L = \sum a_i x^{\sigma^i}\in \L$ and
$\alpha\in \Fqn$, define $L_\alpha = \sum a_i \alpha^{[i]}x^{\sigma^i}$. It is clear that $P_{L_\alpha} (y) = P_L(\alpha y)$. Furthermore it holds that
\[
L_\alpha \circ M_\alpha = (L\circ M)_\alpha,
\]
i.e. $L\mapsto L_\alpha$ is a homomorphism on the ring of linearized polynomials with composition. 
Consequently, we have that
\[
\gcd(L,M)_\alpha = \gcd(L_\alpha,M_\alpha)
\]
(recall that $\gcd = \gcrc$ by Remark \ref{gcdrem}).
\begin{remark}
This arises naturally in the skew-polynomial ring, where the map $\ell(t)\mapsto \ell(\alpha t)$ is a ring-isomorphism.
\end{remark}

\section{Main Results}\label{sectmain}

We present one Proposition before our main theorems.

\subsection{Roots of $L$ and $P_L$}

The following relates the roots of $P_L(y)$ to the roots of the $L_\alpha$'s. 
\begin{proposition}
\label{prop:PLroots1}
The number of roots of $P_L(y)$ in $\Fqn$ is equal to 
\[
\sum_{\lambda\in \Fq^*} \frac{q^{d_\lambda}-1}{q-1}
\quad where \quad d_\lambda=\deg_\sigma(\gcrc(L(x),x^{\sigma^n}-\lambda x)).
\]
\end{proposition}

\begin{proof}
Suppose $t$ is a root of $P_L(y)$ with $N(t)=\lambda$. Let $\alpha$ be a fixed element of $\Fqn$ with $N(\alpha)=\lambda$. Then $t = \alpha z^{\sigma-1}$ for some $z\in \Fqn$, and $z$ is unique up to $\Fq^*$-multiplication (i.e. $t = \alpha z^{\sigma-1}=\alpha w^{\sigma-1}$ if and only if $z/w\in \Fq^*$).

Hence $0=P_L(t) = P_L(\alpha z^{\sigma-1}) = P_{L_\alpha}(z^{\sigma-1}) = L_{\alpha}(z)/z$, and thus the number of roots of $P_L(y)$ in $\Fqn$  having norm $\lambda$
is equal to the number of one-dimensional $\Fq$-subspaces of $\ker(L_\alpha)$, which is 
\[
\frac{q^{\deg_\sigma(\gcd(L_\alpha(x),x^{\sigma^n}-x))} - 1}{q-1}.
\]
Finally, it suffices to show that $\deg_\sigma(\gcd(L_\alpha(x),x^{\sigma^n}-x))=\deg_\sigma(\gcd(L(x), x^{\sigma^n}- \lambda x))$. Note that $(x^{\sigma^n}- x)_{\alpha^{-1}} = N(\alpha^{-1})x^{\sigma^n}-x = \lambda^{-1}x^{\sigma^n}-x$. Hence if 
\begin{align*}
D(x) &= \gcd(L(x), x^{\sigma^n}- \lambda x) \\
&= \gcd(L(x), (x^{\sigma^n}- x)_{\alpha^{-1}})\\
\text{ then } D_\alpha (x)&= \gcd(L_\alpha(x),x^{\sigma^n}-x),
\end{align*}
proving the claimed equality of degrees. Summing over $\lambda\in \Fq^*$ gives the desired result
(recall that $\gcd = \gcrc$ by Remark \ref{gcdrem}).
\end{proof}

Later we will use this to give a more useful formula for calculating the number of roots of $P_L(y)$.

Note that the one-dimensional subspaces of the kernel of $L(x)$ are in one-to-one correspondence with solutions of the system $P_L(y)=0, N(y)=1$, where $N(y) = y^{[n]}$ denotes the field norm from $\Fqn$ to $\Fq$.
%
%
%
%

\subsection{First Main Result}

We are now ready for our main results, which relate the number of roots of $L$ and $P_L$ to properties of the matrix $A_L$.
We emphasize again that the matrix $A_L$ is smaller than the Dickson matrix, although
having the same property.
Indeed, for $L\in \L$ of $\sigma$-degree $d$, the Dickson matrix is $n\times n$ whereas
$A_L$ is $d\times d$.

\begin{theorem}
\label{thm:nullAL}
Let $L\in \L$, and let $A_L$ be defined as in (\ref{eqn:AL}). Then for any $\alpha\in \Fqn^*$ we have
\[
\nul(L_{\alpha}) = \nul(A_L-\lambda I),
\]
where $\lambda = N(\alpha)$.
In particular, the nullity  of $L$ is equal to the nullity of $A_L-I$.
\end{theorem}

\begin{proof}
Suppose $\deg_\sigma(L)=d$. By Proposition \ref{prop:AL}, $A_L$ is the matrix associated to the linear transformation $\phi$ of $\L_d$ defined by
\[
\phi(M) = x^{\sigma^n}\circ M \modr L,
\]
and so $A_L-\lambda I$ is the matrix associated to $\phi-\lambda \cdot  id$, which is defined by
\[
(\phi-\lambda  \cdot  id)(M) = (x^{\sigma^n}-\lambda x)\circ M \modr L.
\]

Suppose $M\in \ker(\phi-\lambda  \cdot  id)$, which occurs if and only if $(x^{\sigma^n}-\lambda x)\circ M$ is a left composition of $L$. As $(x^{\sigma^n}-\lambda x)$ commutes with every element of $\L$
(cf. Proposition \ref{prop:known})
we have that $(x^{\sigma^n}-\lambda x)\circ M=M\circ (x^{\sigma^n}-\lambda x)$. Let $D = \gcrc(L,x^{\sigma^n}-\lambda x)$, and let $A,B$ be such that $L=A\circ D$, $x^{\sigma^n}-\lambda x = B\circ D$. Note that $\gcrc(A,B)=x$. Then
\begin{align*}
M\circ B\circ D &\equiv 0\modr A\circ D\\
\Leftrightarrow M\circ B &\equiv 0\modr A.
\end{align*}
Thus $M\circ B$ is a left composition of both $A$ and $B$, and so is a left composition of $\lclc(A,B)$, which we denote by $C$. Let $E,F$ be such that $C=E\circ A = F\circ B$. By Proposition \ref{prop:known} (iii) and (iv), we have that $\deg_\sigma(C)=\deg_\sigma(A)+\deg_\sigma(B)= \deg_\sigma(F)+\deg_\sigma(B)$, and hence $\deg_\sigma(F)=\deg_\sigma(A)$.

Then there exists a unique $G$ such that $M\circ B = G\circ C$. Now $M\circ B = G\circ C=G\circ F\circ B$, and so $M = G\circ F$. Thus $M\in \ker(\phi-\lambda  \cdot  id)$ if and only if $M$ is a left composition of $F$ of $\sigma$-degree at most $d-1$. 

Now $\deg_\sigma(F)=\deg_\sigma(A)$, and $\deg(A) = d - \deg_\sigma(D)$, implying that $\deg(G)\leq \deg_\sigma(D)-1$. Thus
\[
\dim\ker(\phi-\lambda  \cdot  id) = \deg_\sigma(D),
\]
which by Proposition \ref{prop:known} (see also the proof of Proposition 3)
is equal to $\nul(L_{\alpha})$.
\end{proof}

\subsection{Second Main Result}

\begin{theorem}
\label{thm:PLroots2}
The number of roots of $P_L$ in $\Fqn$ is given by
\[
\sum_{\lambda \in \Fq} \frac{q^{n_\lambda}-1}{q-1}, 
\]
where $n_\lambda$ is the dimension of the eigenspace of $A_L$ with eigenvalue $\lambda$.

The number of roots of $L$ in $\Fqn$ is given by $q^{n_1}$, i.e.,
the size of the eigenspace of $A_L$ corresponding to the eigenvalue $1$.
\end{theorem}

\begin{proof}
This follows immediately from Proposition \ref{prop:PLroots1} and Theorem \ref{thm:nullAL}.
Note that we only sum over eigenvalues in $\Fq$.
\end{proof}

\begin{remark}
In \cite{VoGiZi} it was shown that the number of roots of $P_L$ was related to the eigenspaces of a linear map acting on the vector space of roots of $L$ over the algebraic closure of $\Fqn$. However  there was no explicit way of obtaining this linear map directly from the coefficients of $L$, whereas in this paper $A_L$ does provide such a method.
\end{remark}

\subsection{Third Main Result}

\begin{theorem}
\label{thm:splits}
Let $L\in \L$ have $\sigma$-degree $d$. Then $P_L$ has $\frac{q^d-1}{q-1}$ roots in $\Fqn$ if and only if
\[
A_L=\lambda I
\]
for some $\lambda\in \Fq$. Also, $L$ has $q^d$ roots in $\Fqn$ if and only if
\[
A_L= I.
\]
\end{theorem}

\begin{proof}
By Theorem \ref{thm:PLroots2}, $P_L$ has $\frac{q^d-1}{q-1}$ roots if and only if $A_L$ has an eigenspace of dimension $d$, which occurs if and only if $A_L$ is a multiple of the identity.

Furthermore, $L$ has $q^d$ roots in $\Fqn$ if and only if it has nullity $d$, which by
Theorems \ref{thm:nullAL} and \ref{thm:PLroots2} occurs if and only if $A_L=I$.
\end{proof}

An immediate corollary is the following known result \cite[Theorem 10]{GoQu2009a}, which was used in \cite{SheekeyMRD} to construct new maximum rank-distance codes.
\begin{corollary}\label{cor:norm}
If $L$ is a $\sigma$-linearized polynomial of $\sigma$-degree $d$ such that $L$ has $q^d$ roots, then $N(a_0)=(-1)^{nd}N(a_d)$.
\end{corollary}

\begin{remark}
Note that if $x^\sigma = x^q$, then this theorem gives us criteria for when $L$ or $P_L$ split completely over $\Fqn$.
\end{remark}

In the next sections we will use our main results to 
develop explicit criteria for the number of roots of 
$L$ and $P_L$ in the case of $\sigma$-degree two and three.
Our method will work for any $\sigma$-degree $d$, and any $n$, however the
expressions become more complicated as $d$ increases.

\section{Criteria for the number of roots for  $\sigma$-degree $2$}\label{sectd2}

\subsection{Recursions for the matrix $C_k$}
Let $L = ax+bx^\sigma+cx^{\sigma^2}$ with $a,b,c\in \Fqn$, $ac\ne 0$, and let
\[
C = C_L = \npmatrix{0&-a/c\\1&-b/c}.
\]
We define $C_0$ to be the $k\times k$ identity matrix, and for $k\geq 1$ define
\[
C_k = CC^{\sigma}\cdots C^{\sigma^{k-1}},
\]
and note that  $C_n = A_L$.
Observe that
\[
C_k = C_{k-1}C^{\sigma^{k-1}} = CC_{k-1}^{\sigma}.
\]
These two relations will allow us to obtain straightforward recursions for the entries of $C_k$. We make the following substitutions in order to simplify the calculations. Let
\begin{align*}
u &= \frac{a^\sigma c}{b^{\sigma+1}}
\end{align*}
and define matrices 
\[
X = \npmatrix{a/b&0\\0&1}, \quad Z_k=\npmatrix{(b/c)^{[k-1]}&0\\0&(b/c)^{[k]}}
\]
and define $Y_k$ by $C_k = XY_kZ_k$.


\begin{remark}\label{b0}
Although the introduction of $u$ suggests that we must assume that $b\ne 0$, this is not the case. We work with this expression formally, purely for convenience and ease of notation and calculation. The matrix $A_L$ contains no expressions with a denominator involving $b$, and so the same is true of all determinants and minors calculated in the sequel. Thus it is safe to perform calculations using $u$, so long as one clears denominators before evaluating an expression.
\end{remark}

It is easy to check that
\[
Y_1 = \npmatrix{0&-1\\1&-1},
\]

\begin{lemma}
\label{prop:Yrec}
Let $k\ge 2$, then
$Y_k$ satisfies the recursions
\begin{align*}
Y_k &= Y_{k-1}\npmatrix{0&-u\\1&-1}^{\sigma^{k-2}}\\
Y_k&= \npmatrix{0&-1\\u&-1}Y_{k-1}^\sigma.
\end{align*}
Thus $Y_k$ is a matrix each of whose entries is a polynomial in the indeterminate $u$. 
\end{lemma}

\begin{proof}
We have that 
\[
Y_k =  X^{-1}C_kZ_k^{-1}=X^{-1}C_{k-1}C^{\sigma^{k-1}}Z_k^{-1}
= X^{-1}(XY_{k-1}Z_{k-1})C^{\sigma^{k-1}}Z_k^{-1},
\]
and so the proof of the first follows from a simple verification that $Z_{k-1}C^{\sigma^{k-1}}Z_k^{-1}$ is the stated matrix 
$\npmatrix{0&-u\\1&-1}^{\sigma^{k-2}}$.

Similarly we have
\[
Y_k =X^{-1}CC_{k-1}^\sigma Z_k^{-1}= (X^{-1}CX^\sigma) Y_{k-1}^\sigma (Z_{k-1}^\sigma Z_k^{-1})
\]
and the proof of the second follows from a verification that $Z_{k-1}^\sigma Z_k^{-1}=(c/b)I$, and that $(c/b)X^{-1}CX^\sigma$ is the stated matrix
$\npmatrix{0&-1\\u&-1}$.

%
%
%
%
\end{proof}

%

We will write the second column of $Y_k$ as $(F_k,G_k)^T$, where $F_k$ and $G_k$ are functions of $u$. 
In other words, we define $G_k$ to be the $(2,2)$ entry of $Y_k$, and define
$F_k$ to be the $(1,2)$ entry of $Y_k$.
Then $G_0=1$, $G_1=-1$ and $G_2=1-u$.
It is clear from the first recursion 
in Lemma \ref{prop:Yrec} that
\[
Y_k = \npmatrix{F_{k-1}&F_k\\G_{k-1}&G_k},
\]
and that
\begin{equation}
\label{eqn:R2a}
G_k+G_{k-1}+u^{\sigma^{k-2}}G_{k-2} = 0.
\end{equation}
Furthermore, the second recursion gives us that $F_k = -G_{k-1}^\sigma$, and $G_k = uF_{k-1}^\sigma -G_{k-1}^\sigma$. Together this gives
\begin{equation}
\label{eqn:R2b}
G_k+G_{k-1}^\sigma+uG_{k-2}^{\sigma^2} = 0.
\end{equation}

We observe that these recursions are 
essentially the same as those found by Helleseth-Kholosha in \cite{HellKhol} in the special case $q$ even, although their method was very different from ours.

Therefore we can rewrite $Y_k$ in terms of the sequence $(G_k)$, as follows:
\[
Y_k = \npmatrix{-G_{k-2}^\sigma&-G_{k-1}^\sigma\\
G_{k-1}&G_{k}}
\]

Note that if $k=n$, (\ref{eqn:R2a}) and (\ref{eqn:R2b}) imply that
\[
G_n^{\sigma^2}-G_n = G_{n-1}^\sigma - G_{n-1}^{\sigma^2}.
\]

Thus, as $C_n = A_L$ by definition, we have $A_L = XY_nZ_n$, which shows the following.
\begin{proposition}
\label{prop:ALmat}
Let $L = ax+bx^\sigma+cx^{\sigma^2}$ with $a,b,c\in \Fqn$, $ac\ne 0$, and let $G_k$ be as defined above. Then
\[
A_L =  N(b/c)\npmatrix{-u^{\sigma^{-1}}G_{n-2}^\sigma&-(a/b)G_{n-1}^\sigma\\
(c/b)^{\sigma^{-1}}G_{n-1}&G_{n}}
\]
\end{proposition}

To calculate the number of roots of $L$ and $P_L$, we need to know the characteristic polynomial of $A_L$. 
Directly from Proposition \ref{prop:ALmat} we get
\begin{align*}
\det(A_L) &= N(b/c)^2 u^{\sigma^{-1}}(G_{n-1}^{\sigma+1} - G_nG_{n-2}^\sigma)\\
\Tr(A_L) &= N(b/c)(G_n-u^{\sigma^{-1}}G_{n-2}^\sigma)\\
&= N(b/c)(G_n+G_{n}^{\sigma}+G_{n-1}^\sigma).
\end{align*}
Note that $\det(A_L) = N(\det(C))$ in general, and so for $d=2$ we get
\[
\det(A_L) = N(a/c).
\]
Thus we have that
\[
N(b/c)^2 u^{\sigma^{-1}}(G_{n-1}^{\sigma+1} - G_nG_{n-2}^\sigma) = N(a/c),
\]
and so
\[
u^{\sigma^{-1}}(G_{n-1}^{\sigma+1} - G_nG_{n-2}^\sigma) = N(u).
\]
This proves the following. Recall that $\chi_L$ is the characteristic 
polynomial of $A_L$.

\begin{proposition}
\label{prop:charAL}
Let $L = ax+bx^\sigma+cx^{\sigma^2}$ with $a,b,c\in \Fqn$, $ac\ne 0$, and let $G_k$ be as defined above. Let $\mu = N(b/c)$. Then
\[
\chi_{L}(\mu x) =\mu^2\left(x^2 -  (G_n+G_n^{\sigma}+G_{n-1}^{\sigma})x + N(u)\right).
\]
\end{proposition}


\subsection{Criteria for number of roots of $P_L$}

If $\deg_\sigma(L) = 2$, then $A_L$ is a $2\times 2$ matrix. Thus 
\[
\chi_L(x) = x^2 - \Tr(A_L)x+\det(A_L).
\]
By Theorem \ref{thm:PLroots2}, the number of roots of $P_L$ in $\Fqn$ is determined by the dimension of the eigenspaces of $A_L$ corresponding to eigenvalues in $\Fq$. Hence if $\chi_L(x)$ has two distinct roots in $\Fq$, then $P_L$ has two roots in $\Fqn$. If $\chi_L(x)$ has no roots in $\Fq$, then $P_L$ has no roots in $\Fqn$. If 
$\chi_L(x)$ has a double root in $\Fq$, then $P_L$ has either $1$ or $q+1$ root(s) in $\Fqn$; if $A_L$ is a scalar multiple of the identity then $P_L$ has $q+1$ roots, otherwise it has $1$ root. Hence we have the following.

If $q$ is odd, let 
\begin{align*}
\Delta_L&= \Tr(A_L)^2 - 4\det(A_L)\\
&= N(b/c)^2((G_n+G_n^{\sigma}+G_{n-1}^{\sigma})^2 -4N(u))
\end{align*}

If $q$ is even, let
\begin{align*}
\Lambda_L &=  \det(A_L)/\Tr(A_L)^2
\end{align*}
and let $T_0$ denote the absolute trace, i.e. the trace from $\Fqn$ to $\FF_2$.

\begin{theorem} 
\label{thm:PLcriteria2}
Let $L(x) = ax+bx^\sigma+cx^{\sigma^2}\in \Fqn[x]$. Let $Z_L$ denote the number of solutions to $P_L(x) = a+bx+cx^{\sigma+1}=0$ in $\FF_{q^n}$. For any $q$, $Z_L =q+1$ if and only if $G_{n-1}=0$ and $G_n\in \Fq$. 

If $q$ is odd and $b\ne 0$, then 
\begin{itemize}
\item
$Z_L =q+1$ if and only if $\Delta_L=0$ and $G_{n-1}= 0$.
\item
$Z_L =1$ if and only if $\Delta_L=0$ and $G_{n-1}\ne 0$.
\item
$Z_L =2$ if and only if $\Delta_L$ is a non-zero square in $\Fq$.
\item
$Z_L =0$ if and only if $\Delta_L$ is a non-square in $\Fq$.
\end{itemize}


If $q$ is even and $b\ne 0$, then 
\begin{itemize}
\item
$Z_L =q+1$ if and only if $G_{n-1}=0$ and $G_{n}\in \Fq$.
\item
$Z_L =1$ if and only if $G_{n-1}\ne 0$ and $G_{n}\in \Fq$.
\item
$Z_L =2$ if and only if $G_n\notin \Fq$ and $T_0(\Lambda_L)=0$.
\item
$Z_L =0$ if and only if $G_n\notin \Fq$ and $T_0(\Lambda_L)=1$.
%
%
\end{itemize}


If $b=0$ and $n$ is odd, then $Z_L = 1$ if $N(-a/c)$ is a nonzero square in $\Fq$, and $Z_L=0$ otherwise.

If $b=0$ and $n$ is even, then $Z_L = q+1$ if $N_{q^n:q^2}(-a/c)\in \Fq$, and $Z_L=0$ otherwise.

\end{theorem}

\begin{proof}
By Theorem \ref{thm:splits}, $P_L$ has $q+1$ roots if and only if $A_L$ is a scalar multiple of the identity. By Proposition \ref{prop:ALmat}, $A_L$ is a diagonal matrix if and only if $G_{n-1}= 0$. Suppose $G_{n-1}= 0$. 
Then \eqref{eqn:R2a} and \eqref{eqn:R2b} with $k=n$ are
\[
G_n+ u^{\sigma^{n-2}}G_{n-2} = 0\quad\quad \textrm{and} \quad\quad
G_n+uG_{n-2}^{\sigma^2} = 0.
\]
Thus we get that $G_n=G_n^{\sigma^2}$, and so $G_n\in \FF_{q^2}$. 
Then $-u^{\sigma^{n-1}}G_{n-2}^\sigma=G_n^\sigma$, and so
\[
A_L = N(b/c)\npmatrix{G_n^{\sigma}&0\\0&G_n}
\]
Thus $A_L$ is a scalar multiple of the identity if and only if $G_{n-1}=0$ and $G_n\in \Fq$, as claimed.



The remaining cases follow immediately from Theorem \ref{thm:PLroots2} and Proposition \ref{prop:charAL}, together with the well-known criteria for the number of distinct roots of a quadratic polynomial.
\end{proof}

\subsection{Criteria for number of roots of $L$}

We now present the analogous theorem for the linearized polynomial, which 
classifies all the possibilities for the number of roots.

\begin{theorem}\label{degree2alln}
Let $L(x) = ax+bx^\sigma+cx^{\sigma^2}\in \Fqn[x]$ with $ac\not=0$. Then 
\begin{itemize}
\item[(i)]
$L$ has $q^2$ roots in $\Fqn$ if and only if $G_{n-1}=0$ and $N(b/c)G_n = 1$;
\item[(ii)]
$L$ has $q$ roots in $\Fqn$ if and only if 
$$1 -  N(b/c)( G_n+G_n^{\sigma}+G_{n-1}^{\sigma}) + N(a/c)=0,$$ and the conditions of (i) are not satisfied.
\item[(iii)]
$L$ has $1$ root in $\Fqn$ if and only if 
$$1 -  N(b/c)( G_n+G_n^{\sigma}+G_{n-1}^{\sigma}) + N(a/c)\ne 0.$$
\end{itemize}
\end{theorem}

\begin{proof}  
(i) By Theorem \ref{thm:splits}, $L$ has $q^2$ roots in $\Fqn$ if and only if $A_L=I$. The same argument as in Theorem \ref{thm:PLcriteria2} shows that this occurs if and only if $G_{n-1}=0$ and $N(b/c)G_n = 1$, as claimed.

(ii) By Theorem \ref{thm:PLroots2}
$L$ has $q$ roots in $\Fqn$ if and only if $1$ is an eigenvalue of $A_L$, and $A_L\ne I$.
Evaluating the characteristic polynomial at 1 using Proposition \ref{prop:charAL}
gives the stated expression.

(iii) By Theorem \ref{thm:PLroots2} $L$ has one root in $\Fqn$ if and only if 1 is not an eigenvalue of $A_L$.
\end{proof}
%

\subsection{Examples for small $n$}

The sequence $(G_k)$ begins as follows, starting at $G_0$.
\[
\begin{array}{|c|c|c|c|c|c|c|}
\hline
k&G_k\\

\hline
0&1\\
1&-1\\
2&1-u\\
3&u^\sigma+u-1\\
4&(1-u)^{\sigma^2+1}-u^\sigma\\
5&1-u^{\sigma^3+1}-(1-u)^{\sigma^2+1}-(1-u)^{\sigma^3+\sigma}\\
\hline
\end{array}
\]

We proceed recursively to compute the $G_k$ and $A_L$. The conditions for the number of roots of $P_L$ depend on $\det(A_L)$ and $\Tr(A_L)$. Note that we always have that $\det(A_L) = N(a/c)$. We compute now the expression for $\Tr(A_L)$ for some small values of $n$. We denote the trace function from $\FF_{q^i}$ to $\FF_{q}$ by $\tr_{q^i:q}$, and the trace function from $\Fqn$ to $\Fq$ by $\tr$.

$\mathbf{n=4}$: For $L\in \FF_{q^4}[x]$, we have 
\begin{align*}
A_L &=N(b/c)\npmatrix{-u^{\sigma^{3}}G_{2}^\sigma&-(a/b)G_{3}^\sigma\\
(c/b)^{\sigma^{3}}G_{3}&G_{4}}\\
&=  N(b/c) \npmatrix{u^{\sigma^3}(u^\sigma-1)&(a/b)(1-u^\sigma-u^{\sigma^2})\\ (c/b)^{\sigma^{3}}(u+u^\sigma-1)&(1-u)^{\sigma^2+1}-u^\sigma}
\end{align*}
Furthermore we have
\begin{align*}
\Tr(A_L)&=N(b/c)(1-\tr_{q^4:q}(u)+\tr_{q^2:q}(u^{1+\sigma^2}))\\
%
%
\Delta_L&= \Tr(A_L)^2 - 4\det(A_L)\\
&=N(b/c)^2\left((G_4+G_4^\sigma +G_{3}^\sigma)^2 - 4N(u)\right).
\end{align*}

$\mathbf{n=5}$:
For $L\in \FF_{q^5}[x]$, we have 
\begin{align*}
A_L &=N(b/c)\npmatrix{-u^{\sigma^{4}}G_{3}^\sigma&-(a/b)G_{4}^\sigma\\
(c/b)^{\sigma^{4}}G_{4}&G_{5}}\\
\Tr(A_L)&=N(b/c)(\tr(u-u^{1+\sigma^2})-1)
\end{align*}


$\mathbf{n=6}$: For $L\in \FF_{q^6}[x]$, we have 
\begin{align*}
\Tr(A_L)&=N(b/c)(1-\tr_{q^6:q}(u-u^{1+\sigma^2})+\tr_{q^3:q}(u^{1+\sigma^3}) -\tr_{q^2:q}(u^{1+\sigma^2+\sigma^4})).
\end{align*}

$\mathbf{n=7}$: For $L\in \FF_{q^7}[x]$, we have 
\begin{align*}
\Tr(A_L)&=N(b/c)(\tr(u-u^{\sigma^2+1}-u^{\sigma^3+1}+u^{\sigma^4+\sigma^2+1})-1)
\end{align*}

\section{Criteria for the number of roots for $\sigma$-degree $3$}\label{sectd3}


\subsection{Recursions for the matrix $C_k$}

Let $L = ax+bx^\sigma+cx^{\sigma^2}+dx^{\sigma^3}$ with $a,b,c,d\in \Fqn$, $ad\ne 0$.

Define
\[
C_L = \npmatrix{0&0&-a/d\\1&0&-b/d\\0&1&-c/d},
\]
and 
\[
C_k = C_LC_L^{\sigma}\cdots C_L^{\sigma^{k-1}}.
\]

Note that
\[
C_k = C_{k-1}C_L^{\sigma^{k-1}} = C_LC_{k-1}^{\sigma}.
\]

We define
\[
w = \frac{a^\sigma c}{b^{\sigma+1}}\quad\quad \textrm{and}\quad\quad
z = \frac{b^\sigma d}{c^{\sigma+1}}.
\]

We note the following identities:
\[
wz = \frac{a^\sigma d}{bc^\sigma}\quad\quad\quad\quad
w^\sigma z^{\sigma+1} = \frac{a^{\sigma^2}d^{\sigma+1}}{c^{\sigma^2+\sigma+1}}.
\]

%

Let $k\ge 3$ and let us define 
\[
X = \npmatrix{a/c&0&0\\0&b/c&0\\0&0&1}, \quad
Z_k=\npmatrix{(c/d)^{[k-2]}&0&0\\0&(c/d)^{[k-1]}&0\\0&0&(c/d)^{[k]}}
\]
and define $Y_k$ by  $C_k = XY_kZ_k$. Define also
\[
Y_2 = \npmatrix{0&-1&1\\0&-1&1-wz\\1&-1&1-z}
\]

\begin{lemma}
Let $k\ge 3$, then 
$Y_k$ satisfies the recursions
\[
Y_k = Y_{k-1}\npmatrix{0&0&-w^\sigma z^{\sigma+1}\\1&0&-z^\sigma\\0&1&-1}^{\sigma^{k-3}}
\]
\[
Y_k= \npmatrix{0&0&-1\\wz&0&-1\\0&z&-1}Y_{k-1}^\sigma
\]
\end{lemma}

\begin{proof}
The proof is analogous to Lemma \ref{prop:Yrec}.
\end{proof}

%

Thus $Y_k$ is a matrix each of whose entries is a polynomial in the two variables $w,z$. We will write the third column of $Y_k$ as $(F_k,G_k,H_k)^T$. 
In other words, we define $H_k$ to be the $(3,3)$ entry of $Y_k$, with
 $H_0=1$, $H_1=-1$ and $H_2=1-z$.

It is clear from the first recursion that
\[
Y_k = \npmatrix{F_{k-2}&F_{k-1}&F_k\\G_{k-2}&G_{k-1}&G_k\\H_{k-2}&H_{k-1}&H_k},
\]
and that
\begin{equation}\label{firstrec}
H_k+H_{k-1}+z^{\sigma^{k-2}}H_{k-2}+w^{\sigma^{k-2}}z^{\sigma^{k-2}+\sigma^{k-3}} H_{k-3} = 0.
\end{equation}
Furthermore, the second recursion gives us that $F_k = -H_{k-1}^\sigma$, $G_k = wzF_{k-1}^\sigma-H_{k-1}^\sigma$, and $H_k = zG_{k-1}^\sigma - H_{k-1}^\sigma$. Together this gives
\begin{equation}\label{secondrec}
H_k+H_{k-1}^\sigma+zH_{k-2}^{\sigma^2}+ w^\sigma z^{\sigma+1} H_{k-3}^{\sigma^3} = 0.
\end{equation}
Therefore we can rewrite $Y_k$ in terms of the sequence $(H_k)$, as follows:
\[
Y_k = \npmatrix{-H_{k-3}^\sigma&-H_{k-2}^\sigma&-H_{k-1}^\sigma\\
-wzH_{k-4}^{\sigma^2}-H_{k-3}^\sigma&-wzH_{k-3}^{\sigma^2}-H_{k-2}^\sigma&-wzH_{k-2}^{\sigma^2}-H_{k-1}^\sigma\\
H_{k-2}&H_{k-1}&H_{k}}
\]
Note that when $k=n$,  taking  \eqref{firstrec} to the power of 
$\sigma^3$ and subtracting \eqref{secondrec}  gives
\begin{equation}\label{knrec}
H_n^{\sigma^3} + H_{n-1}^{\sigma^3} + z^\sigma H_{n-2}^{\sigma^3}=
H_n + H_{n-1}^{\sigma} + z H_{n-2}^{\sigma^2}.
\end{equation}

As $C_n = A_L$ by definition, we have $A_L = XY_nZ_n$, which shows the following.

\begin{theorem}
\label{prop:ALmat3}
Let $L = ax+bx^\sigma+cx^{\sigma^2}+dx^{\sigma^3}$ with $a,b,c,d\in \Fqn$, $ad\ne 0$, and let $H_k$ be as defined above. Then $A_L$ is determined by $H_n$, $H_{n-1}$, and $H_{n-2}$. Explicitly,
\[
A_L = N(c/d)\npmatrix{H_n^{\sigma^{-2}}+H_{n-1}^{\sigma^{-1}}+z^{\sigma^{-2}}H_{n-2}&-(a/b)z^{\sigma^{-1}}H_{n-2}^\sigma&-(a/c)H_{n-1}^\sigma\\
(d/c)^{\sigma^{-2}}(H_{n-1}^{\sigma^{-1}}+H_{n-2})&H_{n}^{\sigma^{-1}}+H_{n-1}&-(b/c)(wzH_{n-2}^{\sigma^2}+H_{n-1}^\sigma)\\
(d/c)^{\sigma^{-1}+\sigma^{-2}}H_{n-2}&(d/c)^{\sigma^{-1}}H_{n-1}&H_{n}}.
\]
Furthermore, the coefficients of the characteristic polynomial $\chi = x^3-\chi_2 x^2+\chi_1 x-\chi_0$ of $A_L$ are given by 
\begin{align*}
\chi_0 = \det(A_L) &= N(-a/d)=N(c/d)^3N(-w)N(z)^2,
\end{align*}
and
\begin{align*}
\chi_1  &= N(c/d)^2(a_{1}+a_{2}+a_{3})
\end{align*}
where
\begin{align*}
a_{1} &= w^\sigma z^{\sigma+1}(H_{n-1}^\sigma H_{n-2}^{\sigma^3}+H_{n-2}^{\sigma^2+\sigma^3}+z^\sigma H_{n-2}^{\sigma^3} H_{n-3}^{\sigma^3}+w^{\sigma^2}z^{\sigma+\sigma^2}H_{n-3}^{\sigma^3+\sigma^4})\\
a_{2} &= w^\sigma z^{\sigma+1}(H_{n}^{\sigma^2} H_{n-3}^{\sigma^3}+H_{n-1}^{\sigma^3} H_{n-2}^{\sigma^2})\\
a_{3} &= w^{\sigma^2} z^{\sigma+\sigma^2}(H_{n}^{\sigma^2} H_{n-3}^{\sigma^4}+H_{n-1}^{\sigma^2} H_{n-2}^{\sigma^4})+z^\sigma (H_{n}^{\sigma^2} H_{n-2}^{\sigma^3}+H_{n-1}^{\sigma^2+\sigma^3}),
\end{align*}
and
\begin{align*}
\chi_2 = \Tr(A_L) &= N(c/d)(H_n^{\sigma^{2}}+H_{n}^{\sigma}+H_{n}+H_{n-1}^{\sigma}+
H_{n-1}^{\sigma^{2}}+zH_{n-2}^{\sigma^{2}}).
\end{align*}
\end{theorem}

\begin{proof}
It can be easily verified using the matrix $Y_k$ that $A_L=XY_nZ_n$ has the given form. The proof of the coefficients $\chi_1$ and $\chi_2$ then follows by direct calculation, taking into account the two recursions and the fact that each coefficient is in $\Fq$ by Theorem \ref{lem:charq}.
For example, 
we have taken the trace of the matrix to the power of ${\sigma^{2}}$. 

The constant term $\chi_0$ is equal to the determinant of $A_L$, which is equal to the norm of the determinant of $C_L$ (see Remark \ref{detnorm})
and the determinant of $C_L$ is $-a/d$.
\end{proof}


\subsection{Criteria for number of roots of $P_L$}

\begin{theorem}\label{d3split}
If $L= ax+bx^\sigma+cx^{\sigma^2}+dx^{\sigma^3}$ with $ad\ne 0$, then 
the number of roots of $P_L=a+bx+cx^{\sigma +1}+dx^{\sigma^2+\sigma +1}$ in 
$\Fqn$ is given by the
eigenvalues and eigenvectors of $A_L$, as follows.

Let $m_a$ denote the algebraic multiplicity and let $m_g$ denote the 
geometric multiplicity of an eigenvalue of $A_L$.
\begin{itemize}
\item $P_L$ has $q^2+q+1$ roots in $\Fqn$ if and only if $A_L$ has one
eigenvalue in $\Fq$ with multiplicity pair $(m_a,m_g)=(3,3)$.
\item $P_L$ has $q+1$ roots in $\Fqn$  if and only if $A_L$ has one
eigenvalue in $\Fq$ with multiplicity pair $(m_a,m_g)=(3,2)$.
\item $P_L$ has $1$ root in $\Fqn$  if and only if $A_L$ has one
eigenvalue in $\Fq$ with multiplicity pair $(m_a,m_g)=(3,1)$ or $(1,1)$.
\item $P_L$ has $q+2$ roots in $\Fqn$  if and only if $A_L$ has one
eigenvalue in $\Fq$ with multiplicity pair $(m_a,m_g)=(1,1)$
and one eigenvalue in $\Fq$ with multiplicity pair $(m_a,m_g)=(2,2)$
\item $P_L$ has $2$ roots in $\Fqn$  if and only if $A_L$ has one
eigenvalue in $\Fq$ with multiplicity pair $(m_a,m_g)=(1,1)$
and one eigenvalue in $\Fq$ with multiplicity pair $(m_a,m_g)=(2,1)$
\item $P_L$ has $3$ roots in $\Fqn$  if and only if $A_L$ has three
distinct
eigenvalues in $\Fq$ each with multiplicity pair $(m_a,m_g)=(1,1)$
\item $P_L$ has $0$ roots in $\Fqn$  if and only if $A_L$ has 
no eigenvalues in $\Fq$.
\end{itemize}
There are no other possibilities.
\end{theorem}

\begin{proof}
The follows from Theorem  \ref{thm:PLroots2} and
consideration of all possible eigenvalue-eigenvector behaviours.
Only eigenvalues in $\Fq$ contribute, and the characteristic polynomial
must have $0, 1$ or 3 roots in $\Fq$ because it has coefficients in $\Fq$ by Theorem \ref{lem:charq}.
\end{proof}

Exact criteria for each case can be found in a similar way to the 
$\sigma$-degree 2 case.
As an example, we give the exact criteria for the first case,
when $P_L$ has the maximum number of roots in $\Fqn$.

\begin{theorem}\label{d3split}
If $L= ax+bx^\sigma+cx^{\sigma^2}+dx^{\sigma^3}$ with $abcd\ne 0$, then 
$P_L=a+bx+cx^{\sigma +1}+dx^{\sigma^2+\sigma +1}$ has $q^2+q+1$ roots if and only if $H_{n-1}=H_{n-2}=0$ and $H_n\in \Fq$. 
\end{theorem}

\begin{proof}
By Theorem \ref{thm:splits}, $P_L$ has $q^2+q+1$ roots if and only if $A_L$ is a scalar multiple of the identity.

If  $H_{n-1}=H_{n-2} = 0$ then
using the second recursion \eqref{secondrec}, 
this automatically  implies that $-wzH_{n-4}^{\sigma^2}-H_{n-3}^\sigma=0$. 
 Hence, by  the form of the matrix in the statement of Theorem \ref{prop:ALmat3},
$A_L$ is a diagonal matrix if and only if $H_{n-1}=H_{n-2} = 0$.

Suppose now $H_{n-1}=H_{n-2} = 0$. Then
\eqref{knrec} implies that $H_n = H_n^{\sigma^3}$, so $H_n\in \FF_{q^3}$,
and by Theorem \ref{prop:ALmat3}, we get
\begin{align*}
A_L&= (c/d)^{[n]}\npmatrix{H_n^\sigma&0&0\\
0&H_n^{\sigma^2}&0\\
0&0&H_n}.
\end{align*}
Hence $A_L$ is a scalar multiple of the identity if and only if $H_n\in \Fq$.
\end{proof}

\begin{remark}
Theorem \ref{d3split} is valid only for the case where all coefficients are nonzero, due to the fact that the expressions $H_{k}$ are functions of $w$ and $z$, which are not defined for $bc=0$. However the matrix in Theorem \ref{prop:ALmat3} is valid, so long as we clear denominators before evaluating. For example, for $A_L$ to be a scalar multiple of the identity, we require $N(c/d)(a/c)H_{n-1}^\sigma = 0$. For $n=4$ this becomes $N(c/d)(a/c)(z^\sigma +z-1 -w^\sigma z^{1+\sigma})^\sigma=0$, and multiplying out we get no denominators involving $b$ or $c$.

We state full criteria only for the case $abcd\ne 0$ for simplicity. All other cases can be similarly analysed by a more careful consideration of Theorem \ref{prop:ALmat3}.
\end{remark}

\subsection{Criteria for number of roots of $L$}

\begin{theorem}\label{d3split2}
For $g\in \{0,1,2,3\}$,
if $L= ax+bx^\sigma+cx^{\sigma^2}+dx^{\sigma^3}$ with $d\ne 0$, then 
$L$ has $q^g$ roots in $\Fqn$ if and only if $1$ is an eigenvalue of $A_L$
with geometric multiplicity $g$.
\end{theorem}

\begin{proof}
The follows from Theorem  \ref{thm:PLroots2}.
\end{proof}

One can calculate exact criteria for each case
$g=0,1,2,3$ from the matrix $A_L$.
As an example, we state the $g=3$ and $g=0$ cases here. 

\begin{theorem}\label{d3split3}
If $L= ax+bx^\sigma+cx^{\sigma^2}+dx^{\sigma^3}$ with $abcd\ne 0$, then 
$L$ has $q^3$ roots if and only if $H_{n-1}=H_{n-2}=0$ and $H_nN(c/d)=1$.
\end{theorem}

\begin{proof}
By the proof of Theorem \ref{d3split}
$A_L$ is a scalar multiple of the identity if and only if 
$H_{n-1}=H_{n-2}=0$ and
$H_n\in \Fq$, in which case $A_L$ is equal to the identity if and only if $H_nN(c/d)=1$. 
\end{proof}

\begin{theorem}\label{d3split3}
If $L= ax+bx^\sigma+cx^{\sigma^2}+dx^{\sigma^3}$ with $d\ne 0$, then 
$L$ is a permutation polynomial if and only if 
$1-\chi_2+\chi_1-\chi_0\ne 0$, where $\chi_i$ are given in Theorem \ref{prop:ALmat3}.
\end{theorem}

\begin{proof}
$L$ is a permutation polynomial if and only if $0$ is its only root. This occurs if and only if $1$ is not an eigenvalue of $A_L$. The given expression is just the evaluation of the characteristic polynomial of $A_L$ at $1$, and so the result follows.
\end{proof}

%
%

\subsection{Example for $n=5$}
Using the definitions and recursions, we see that the sequence $(H_k)$ begins as follows.

\[
\begin{array}{|c|c|c|c|c|c|c|}
\hline
k&H_k\\
\hline
0&1\\
1&-1\\
2&1-z\\
3&z^\sigma +z-1 -w^\sigma z^{1+\sigma}\\
4&1 - z-z^{\sigma}  -z^{\sigma^2} +w^{\sigma^2} z^{\sigma+\sigma^2} +z^{1+\sigma^2} +w^\sigma z^{1+\sigma}\\
\hline
\end{array}
\]
Thus for a linearized or projective polynomial of $\sigma$-degree $3$ to have maximum number of roots in $\FF_{q^5}$, we require that $H_3=H_4=0$. Thus after some simple algebra we have that this occurs precisely when $z^{\sigma^2+1}+z^\sigma-1=0$ and $z^\sigma +z -1-w^\sigma z^{1+\sigma}=0$. When both of these are satisfied, we have $H_5 = (1-z-z^\sigma)(1-z)^{\sigma^3}$, and further manipulation gives that $H_5= -N(z)$. This immediately gives the following.

\begin{theorem}
Suppose $L= ax+bx^\sigma+cx^{\sigma^2}+dx^{\sigma^3}$ has coefficients in $\FF_{q^5}$, and $abcd\ne 0$. 
Then $P_L=a+bx+cx^{\sigma +1}+dx^{\sigma^2+\sigma +1}$ has $q^2+q+1$ roots
in $\FF_{q^5}$ if and only if $z^{\sigma^2+1}+z^\sigma-1=0$ and $z^\sigma +z -1-w^\sigma z^{1+\sigma}=0$.

Also, $L$ has $q^3$ roots in $\FF_{q^5}$ if and only if $z^{\sigma^2+1}+z^\sigma-1=0$, $z^\sigma +z -1-w^\sigma z^{1+\sigma}=0$, and $N(z)N(c/d)= -1$.
\end{theorem}


\section{Applications}\label{sectapp}

In this section we outline a few topics where our results could
be applied.

\subsection{MRD codes from Linearized Polynomials}

A $\sigma$-linearized polynomial naturally defines an $\Fq$-linear map from $\Fqn$ to itself. Conversely, for every such map there exist a unique $\sigma$-linearized polynomial of degree at most $n-1$ representing it. By Proposition \ref{prop:known}, every nonzero element of $\L_{k-1}$ has rank at least $n-k+1$, and $\L_{k-1}$ has dimension $nk$ over $\Fq$. Such a set of maps is called a {\it maximum rank-distance code (MRD code)} and optimal with respect to dimension for a given minimum rank. This construction is known as a (generalised) Gabidulin code, though the construction with $a^\sigma=a^q$ is due to Delsarte. We refer to \cite{SheekeyMRD} for background.

Corollary \ref{cor:norm} was first proven in \cite[Theorem 10]{GoQu2009a}, and has been used in for example \cite{SheekeyMRD} and \cite{TrZhHughes} to construct new families of MRD codes.

Further constructions have been obtained in the case $n=6,8$ \cite{CsMaPoZa}, \cite{CsMaZu}, again using properties of $\sigma$-linearized polynomials.

In order to improve on these results, or to classify such objects, more exact criteria for the rank of a $\sigma$-linearized polynomial are required. This is one of the motivations for this paper.

\subsection{Hasse-Witt matrices}

Let $C$ be a smooth projective curve of genus $g>0$ that is defined over $\F_q$.
 The Hasse-Witt matrix $H$ represents the action of the Frobenius operator 
on the cohomology group $H^1(C,\mathcal O_C)$ where $\mathcal O_C$ is the structure sheaf.
The Frobenius operator is a semilinear map analogous to $\psi$ in
Proposition \ref{prop:AL}, and its $n$-th power is a linear map 
analogous to $\phi$ in Proposition \ref{prop:AL}. 
Since $H^1(C,\mathcal O_C)$ is $g$-dimensional the Hasse-Witt matrix is a 
$g\times g$ matrix.

The Hasse-Witt invariant of the curve $C$, also known as the $p$-rank of the curve, 
is  the rank of the matrix
 \[
HH^\sigma H^{\sigma^2} \cdots H^{\sigma^{g-1}}.
\] 
 The similarity between this matrix and  our matrix $A_L$ is clear (we have $H$ instead of $C_L$).
 The Hasse-Witt invariant is therefore also equal to
 the rank of the Frobenius operator composed with itself $g$ times.

  There are papers in the literature  that calculate the Hasse-Witt invariant
 for certain types of curves, by finding the rank of the matrix
 $HH^\sigma H^{\sigma^2} \cdots H^{\sigma^{n-1}}$.
 Finding the rank and characteristic polynomial of these matrices
 therefore has applications in algebraic geometry.
 The results and methods of this paper are applicable to any
  curve whose Hasse-Witt matrix has the form of
  $C_L$  in Section \ref{comp}, i.e., the form of a companion matrix.

\subsection{Cryptography}

Recent work on the Discrete Logarithm Problem (DLP)  in finite fields relies on 
a supply of projective polynomials that split completely,
i.e., they have all their roots in the ground field (see \cite{crypto} for details).
The projective polynomials $x^{q+1}+Bx+B$ were used
in setting new world records for the DLP
because of the high number of $B$'s that exist where the 
polynomial splits completely.

These results about $x^{q+1}+Bx+B$
can be seen as a special case of our $d=2$ results in Section \ref{sectd2},
Theorem \ref{thm:PLcriteria2},
which give precise if and only if conditions for a projective polynomial
$P_L(x) = a+bx+cx^{\sigma+1}\in \Fqn[x]$ to split completely.

Moreover, Theorem \ref{d3split}
presents precise if and only if conditions for a projective polynomial
$P_L(x) = a+bx+cx^{\sigma+1}+dx^{\sigma^2+\sigma+1}\in \Fqn[x]$ to split completely.
It is possible that Theorem \ref{d3split} and
these projective polynomials, or a certain subset of these polynomials,  could be used
to speed up attacks on the finite field DLP.

\section{Comment}
Some results of this paper, in particular the second statement of Theorem \ref{thm:splits}, were simultaneously and independently obtained by the authors of \cite{CsMaPoZuNEW}. The results of this paper and those of \cite{CsMaPoZuNEW} were each presented at the conference {\it Combinatorics 2018} on June 7th 2018.


\begin{thebibliography}{10}
%
%
%
%
%
%
%
%
%
%
%
%
%
%
%
%
%
%
%
%
%
%
%
%
%
%
%
%
%
%
%
%
%
%
%
%
%
%
%
%
%
%
%
%
%
%
%
%

\bibitem{Abhyankar}
Abhyankar, S.S.: {\it Projective polynomials}, Proc. Amer. Math. Soc. 125 (1997) 1643-1650. 


\bibitem{Bluher}
Bluher, A.: {\it On $x^{q+1}+ax+b$}, Finite Fields Appl. 10 (2004) 285-305.

\bibitem{CsMaPoZa}
Csajb\'ok, B.; Marino, G.; Polverino, O.; Zanella C.; {\it A new family of MRD-codes}, 
Linear Algebra and its Applications,
Volume 548, 1, July 2018,  203-220.


\bibitem{CsMaPoZuNEW}
Csajb\'ok, B.; Marino, G.; Polverino, O.; Zullo, F.:
{\it A characterization of linearized polynomials with maximum kernel}, 
Finite Fields and Their Applications, Volume 56, March 2019,  109--130.

\bibitem{CsMaZu}
Csajb\'ok, B.; Marino, G.; Zullo, F.:
{\it New maximum scattered linear sets of the projective line}, 
Finite Fields and Their Applications 54, 133-150, 2018. 13, 2018.

\bibitem{VoGiZi}
von zur Gathen, J.; Giesbrecht, M; Ziegler, K.: {\it Composition collisions and projective polynomials}. ISSAC (2010) 123-130.

\bibitem{HellKhol}
Helleseth, T.; Kholosha, A. {\it  $x^{2^l+1}+x+a$ and related affine polynomials over $GF(2^k)$}. Cryptogr. Commun. 2 (2010) 85-109.


\bibitem{Giesbrecht}
Giesbrecht, Mark: {\it Factoring in skew-polynomial rings over finite fields}. J. Symbolic Comput. 26 (1998), no. 4, 463--486.

\bibitem{crypto}
Gologlu, F., Granger, R., McGuire, G., Zumbragel, J.: {\it On the function field sieve and the impact of higher splitting probabilities: application to discrete logarithms in $\FF_{2^{1971}}$ and $\FF_{2^{3164}}$}, Advances in cryptology CRYPTO 2013. Part II, 109-128..

\bibitem{GoQu2009a} 
Gow, R., Quinlan, R.: {\it Galois theory and linear algebra}, Linear Algebra Appl. 430 (2009), 1778-1789.

\bibitem{Jacobson}
Jacobson, N.:{\ Finite-dimensional division algebras over fields}, Springer-Verlag, Berlin, 1996.


\bibitem{LN} 
Lidl, R.; Niederreiter, H.; {\it Finite Fields},
Encyclopedia of Mathematics and its Applications, Addison-Wesley, Reading, Mass. (1983).

\bibitem{LeBorgne}
Le Borgne, J.: {\it Semi-characteristic polynomials, $\phi$-modules and skew polynomials}, PhD thesis, arXiv:1105.4083 

\bibitem{Ore}
Ore, O.: {\it On a special class of polynomials}, Trans. Amer. Math. Soc. 35 (1933) 559-584.


\bibitem{SheekeyMRD}
Sheekey, J.; {\it A new family of linear maximum rank distance codes}, Adv. Math. Commun.  10 (2016) 475-488.

\bibitem{S2} Sheekey, J,; New semifields and new MRD codes from skew polynomial rings,
arxiv.org 1806.00251

\bibitem{TrZhHughes}
Trombetti, R.; Zhou, Y.: {\it A new family of MRD codes in $\Fq^{2n\times 2n}$ with right and middle nuclei $\Fqn$}, 
to appear in IEEE Trans. Inf. Th.


\bibitem{WuLiu}
Wu, B.; Liu, Z.: {\it Linearized polynomials over finite fields revisited}, Finite Fields Appl. 22 (2013) 79-100.

%
\end{thebibliography}
\end{document}